\documentclass{amsart}
\usepackage{stmaryrd}
\usepackage{silence}
\usepackage{amssymb}
\usepackage{faktor}
\usepackage[all,cmtip]{xy}
\usepackage{amsmath} 
\usepackage{mathrsfs}
\usepackage{tikz}
\usepackage{tikz-cd}
\usepackage{enumitem}
\usepackage{mathtools}
\usepackage[mathcal]{euscript}
\usepackage{graphicx}
\usepackage{url}
\usepackage{hyperref}
\usepackage[T1]{fontenc}
\usepackage{subfigure}

\usetikzlibrary{shapes.geometric}
\usetikzlibrary{decorations.markings}
\usetikzlibrary{patterns,decorations.pathreplacing}

\usepackage{ragged2e}

\definecolor{coloryellow}{RGB}{240,228,66}
\definecolor{colorskyblue}{RGB}{86,180,233}
\definecolor{colorvermillion}{RGB}{213,94,0}

\newcommand{\graphfont}{\mathsf}

\newcommand{\thetagraph}[1]{\graphfont{\Theta}_{#1}}
\newcommand{\completegraph}[1]{\graphfont{K}_{#1}}

\newcommand{\stargraph}[1]{\graphfont{S}_{#1}}
\newcommand{\graf}{\graphfont{\Gamma}}

\newcommand{\Z}{{\mathbb {Z}}} 
\newcommand{\R}{{\mathbb {R}}} 
\newcommand{\Q}{{\mathbb {Q}}} 
\newcommand{\cat}{{\sf {cat}}} 
\newcommand{\tc}{{\sf {TC}}} 

\theoremstyle{definition}
\newtheorem{definition}{Definition}[section]

\newtheorem{notation}[definition]{Notation}

\newtheorem{example}[definition]{Example}

\newtheorem{construction}[definition]{Construction}
\newtheorem{observation}[definition]{Observation}

\theoremstyle{plain}
\newtheorem{proposition}[definition]{Proposition}

\newtheorem{corollary}[definition]{Corollary}
\newtheorem{theorem}[definition]{Theorem}

\theoremstyle{definition}\newtheorem{problem}{Problem}

\theoremstyle{remark}
\newtheorem{remark}[definition]{Remark}

\makeatletter
\@addtoreset{definition}{section}
\makeatother

\usepackage{marginnote}
    \DeclareFontFamily{U}{wncy}{}
    \DeclareFontShape{U}{wncy}{m}{n}{<->wncyr10}{}
    \DeclareSymbolFont{mcy}{U}{wncy}{m}{n}
    \DeclareMathSymbol{\Sha}{\mathord}{mcy}{"58}

\newsavebox{\foobox}

\title{Farber's conjecture and beyond}
\author{Ben Knudsen}
\email{b.knudsen@northeastern.edu}
\address{Department of Mathematics, Northeastern University, Boston, MA 02115, USA}

\begin{document}

\begin{abstract}
We survey two decades of work on the (sequential) topological complexity of configuration spaces of graphs (ordered and unordered), aiming to give an account that is unifying, elementary, and self-contained. We discuss the traditional approach through cohomology, with its limitations, and the more modern approach through asphericity and the fundamental group, explaining how they are in fact variations on the same core ideas. We close with a list of open problems in the field.
\end{abstract}

\maketitle

\section{Introduction}

In real world situations, motion planning is often motion planning on a graph. One thinks of a self-driving car navigating a network of city streets---and with this thought comes traffic. Motion planning in this setting must be safe, which is to say collision-free. In this way, we are led to contemplate motion planning in the ordered \emph{configuration space}
\[F_k(\graf)=\{(x_1,\ldots, x_k)\in\graf^k:x_i\neq x_j\text{ if } i\neq j\}\]
of $k$ particles on the graph $\graf$. These spaces and their cousins, the unordered configuration spaces $B_k(\graf):=F_k(\graf)/\Sigma_k$, first appeared as mathematical objects in the pioneering work of Abrams, Ghrist, {\'{S}}wi\k{a}tkowski, and their respective collaborators \cite{Abrams:CSBGG,Swiatkowski:EHDCSG,Ghrist:CSBGGR,GhristKoditschek:SFRDB,AbramsLandauLandauPommersheimZaslow:ERWCD}. They have received a great deal of attention in the intervening two decades, some of which is represented in the text and references below.

In view of their emergence from considerations of collision-free motion planning, it is no surprise that configuration spaces of graphs soon became objects of study in the emerging effort to quantify the difficulty of motion planning using topological methods---what has now become the well-established subfield of topological complexity \cite{Farber:TCMP}. 

In his seminal study \cite{Farber:CFMPG} of the topological complexity of these spaces, Farber formulated a bold and surprising conjecture, namely that the complexity is eventually independent of $k$, the number of particles---that, for example, it is in a sense precisely as difficult to plan the motion of a billion agents in a network with five nodes as it is to plan the motion of ten. More precisely, Farber's conjecture is the $r=2$ case of the following result, which was proven by the author in \cite{Knudsen:TCPGBGSM}. We write $m(\graf)$ for the number of essential (valence at least three) vertices of $\graf$.

\begin{theorem}\label{thm:ordered}
Let $\graf$ be a connected graph with $m(\graf)\geq2$, and fix $r>0$. For $k$ at least $2m(\graf)$, we have the equality 
\[\frac{1}{r} \tc_r(F_k(\graf))=m(\graf).\]
\end{theorem}

As has long been known, the quantity $m(\graf)$ is a universal upper bound for the homotopical dimension of $F_k(\graf)$, which is realized for $k\geq 2m(\graf)$ \cite{Swiatkowski:EHDCSG}, so the theorem can be understood as asserting that $\tc_r$ is maximal in the same range.

As expanded upon below, various cases of Theorem \ref{thm:ordered} preceded the definitive result, namely the case of trees and $r=2$, due to Farber \cite{Farber:CFMPG}; the case of ``fully articulated'' and ``banana'' graphs and $r=2$, due to L\"{u}tgehetmann--Recio-Mitter in \cite{LuetgehetmannRecio-Mitter:TCCSFAGBG}; the case of trees and $r>2$, due to Aguilar-Guzm\'{a}n--Gonz\'{a}lez--Hoekstra-Mendoza \cite{Aguilar-GuzmanGonzalezHoekstra-Mendoza:FSMTMHTCOCST}; and the case of planar graphs and all $r$, due to the author \cite{Knudsen:FCPG}.

Rather surprisingly, the obvious unordered analogue of Farber's conjecture is false---even for linear trees!\footnote{An explicit example is given in \cite{LuetgehetmannRecio-Mitter:TCCSFAGBG}, the point being that the braid groups of a linear tree are right-angled Artin groups \cite{ConnollyDoig:BGRAAG}, whose topological complexity is known \cite{CohenPruidze:MPT}.} Apart from scattered results \cite{Scheirer:TCUCSCG,HoekstraMendoza:BHTCCST}, all confined to the setting of trees, the topological complexity of unordered configuration spaces of graphs remained largely mysterious until the following result, proven by the author in \cite{Knudsen:OSTCGBG}. 

\begin{theorem}\label{thm:unordered}
Let $\graf$ be a connected graph with $m(\graf)\geq2$, and fix $r>1$. For $k$ at least $2m(\graf)$ plus the number of trivalent vertices, we have the equality 
\[\frac{1}{r} \tc_r(B_k(\graf))=m(\graf)\] provided $\graf$ has no non-separating trivalent vertices.
\end{theorem}

Here, we use the adjective ``non-separating'' to refer to a vertex whose complement in $\graf$ is connected. A separating vertex is otherwise known as an articulation (see Section \ref{subsection:articulation}).

Thus, the exact analogue of Farber's conjecture holds for graphs without trivalent vertices, and its failure can be repaired for most graphs by adding a few more particles. Whether the same is true for graphs with non-separating trivalent vertices is probably the most interesting open problem in the area---see Section \ref{section:problems}.

\begin{remark}
We comment on a few edge cases of Theorems \ref{thm:ordered} and \ref{thm:unordered}. If $m(\graf)=0$, then $\graf$ is homeomorphic either to an interval or to a circle. Thus, $F_k(\graf)$ is homotopy equivalent to a disjoint union of $k!$ or $(k-1)!$ such spaces, respectively, while $B_k(\graf)$ is homotopy equivalent to $\graf$ for every $k>0$. The calculation of $\tc_r$ in these cases is trivial.

If $m(\graf)=1$, then $\graf$ is homeomorphic to the cone on the discrete space $\{1,\ldots, n\}$ for some $n>2$. In this case, $F_k(\graf)$ is homotopy equivalent to a bouquet of circles of cardinality $1+(nk-2k-n+1)\frac{(k+n-2!)}{(n-1)!}$ (see \cite{Ghrist:CSBGGR}, for example). With the single exception of $(k,n)=(2,3)$, one readily checks that this number is at least $2$, so the equality $\tc_r=r$ holds (note that trivalency rears its head again here). The same considerations are valid in the unordered case.

Finally, for $r=1$, it is classical that the conclusion of Theorem \ref{thm:unordered} holds in the improved range $k\geq 2m(\graf)$ (see Corollary \ref{cor:cat}); indeed, we have the a priori stronger fact that the homotopical dimension of $B_k(\graf)$ is $m(\graf)$ in this range \cite{Swiatkowski:EHDCSG}. Thus, the presence of trivalent vertices delays the onset of maximality of $\tc_r(B_k(\graf))$ for $r>1$.
\end{remark}

\begin{remark}
After this article was completed, there appeared the preprint \cite{JankiewiczSchreve:PFGIGBG}, in which the authors claim to remove the proviso ``non-separating'' from Theorem \ref{thm:unordered}. We do not attempt a treatment of this very recent work, save to mention that it crucially uses the results of \cite{CripWiest:EGBSGRAAGBG} relating graph braid groups to right-angled Artin groups.
\end{remark}

\subsection{Outline} Our goal in this survey is to highlight the substantial common conceptual core in essentially all approaches to the topological complexity of configuration spaces of graphs. We hope to articulate this core in a way that is clear, elementary, and self-contained. What novelty there is beyond presentation is modest. Everything should be comprehensible to a reader with a grasp of basic algebraic topology and the fundamentals of Lusternik--Schnirelmann category \cite{James:CSLS} and (sequential) topological complexity \cite{Farber:TCMP,BasabeGonzalezRudyakTamaki:HTCS}.

We begin in Section \ref{section:tori} by considering a general strategy for bounding such invariants from below, which relies on probing the background space using maps from a torus, and which comes in two flavors, homological and homotopical. With this framing, we hope to make clear the close relationship between the earlier (co)homological approaches to Farber's conjecture and the ultimately definitive approach via the fundamental group. In Section \ref{section:configuration spaces}, we build and study the requisite maps from tori into the configuration spaces of interest. In doing so, the author has endeavored to make do with as little as possible---and has succeeded at least in amusing himself. Section \ref{section:detection} gives a brief tour of the applications by various authors of the homological version of the strategy of Section \ref{section:tori} to the tori of Section \ref{section:configuration spaces}, leading to various partial results. After a sobering assessment of the obstructions to further progress, we renounce homology and are carried to victory by the fundamental group. Finally, in Section \ref{section:problems}, we recommend a number of open problems to the eager and industrious reader.

\subsection{Conventions}
A graph is a finite CW complex of dimension at most $1$. The degree $d(v)$ or valence of a vertex $v$ is the number of connected components of its complement in a sufficiently small neighborhood. We say that $v$ is essential if $d(v)\geq3$.

Given an equivalence relation $\pi$ on a set $I$, we write $I/\pi$ for the set of equivalence classes and $[i]$ for the equivalence class of $i\in I$. Partitions of numbers may have empty blocks, and their blocks are distinguishable.

We adhere to the reduced convention on (sequential) topological complexity, according to which $\tc_r(\mathrm{pt})=0$. We further adhere to the convention $\tc_1=\cat$.

\subsection{Acknowledgements} The author learned most of what he knows about configuration spaces of graphs from and with Byunghee An and Gabriel Drummond-Cole. The work of his that he discusses here was nurtured at various stages by the hospitality of the American Institute of Mathematics, the Max Planck Institute for Mathematics, and Casa Matem\'{a}tica Oaxaca, and it was supported by NSF award DMS-1906174. He thanks the referee for a close and constructive reading.

\section{Decomposable tori}\label{section:tori}

In this section, we detail a simple strategy for bounding topological complexity and its cousins from below, the key players being maps from tori into the background space. The main results will follow by applying this strategy to certain tori constructed in Section \ref{section:configuration spaces}.

The strategy comes in two variants, the first of which is (co)homological in nature and goes back at least to \cite{Farber:CFMPG}. For many years, it powered essentially all progress on the (sequential) topological complexity of configuration spaces of graphs \cite{LuetgehetmannRecio-Mitter:TCCSFAGBG,Scheirer:TCUCSCG,Knudsen:FCPG,Aguilar-GuzmanGonzalezHoekstra-Mendoza:FSMTMHTCOCST,HoekstraMendoza:BHTCCST}. The second variant, valid in the aspherical context, deals directly with the fundamental group without the intermediary of (co)homology. Relying crucially on \cite{GrantLuptonOprea:NLBTCAS,FarberOprea:HTCAS}, this strategy led to the proof of the full conjecture \cite{Knudsen:TCPGBGSM} and to the first substantive progress in the unordered setting for graphs with positive first Betti number \cite{Knudsen:OSTCGBG}.

\subsection{Homological decomposability} Throughout this section, (co)homology is taken with integer coefficients (this choice is not essential). We write $T^n=(S^1)^n$ for the standard $n$-torus and $H^*(T^n)=\Lambda[x_1,\ldots, x_n]$ for its cohomology ring.

\begin{definition}
Let $X$ be a topological space.
\begin{enumerate}
\item An $n$-\emph{torus in $X$} is a map $f:T^n\to X$. 
\item We say that an $n$-torus in $X$ is \emph{homologically decomposable} if the homomorphism induced on $H_1$ is injective.
\item We say that a pair of $n$-tori in $X$ are \emph{homologically disjoint} if the images of the two homomorphisms induced on $H_1$ intersect trivially.
\end{enumerate}
\end{definition}

We learned the following argument from \cite{Farber:CFMPG} for $r=2$ and from \cite{Aguilar-GuzmanGonzalezHoekstra-Mendoza:FSMTMHTCOCST} for $r>2$. 

\begin{proposition}\label{prop:decomposable}
Let $X$ be a topological space.
\begin{enumerate}
\item If $X$ admits a homologically decomposable $n$-torus, then $\cat(X)\geq n$. 
\item If $X$ admits a homologically disjoint pair of such, then $\tc_r(X)\geq rn$ for every $r>1$.
\end{enumerate}
\end{proposition}

The proof of the second claim will rely on the production of certain zero-divisors. The negative sign in the following construction leads to various signs in the ensuing argument, which we suppress in order not to distract from the main point. Alternatively, the reader may choose (along with much of the literature) to work over $\mathbb{F}_2$

\begin{notation}
Given a cohomology class $x\in H^*(Y)$, we write $\zeta(x)=x\otimes 1-1\otimes x$, regarded as a cohomology class in $H^*(Y\times Y)$. More generally, fixing $r>1$ and $1\leq a<b\leq r$, we write $\zeta^{ab}(x)$ for the pullback of $\zeta(x)$ along the projection $Y^r\to Y\times Y$ onto the $a$th and $b$th factors.
\end{notation}

Note that the class $\zeta^{ab}(x)$ constructed above is an $r$-fold zero-divisor.

\begin{proof}[Proof of Proposition \ref{prop:decomposable}]
Let $f$ be a decomposable $n$-torus in $X$. By the universal coefficient theorem and decomposability, the homomorphism $f^*:H^1(X)\to H^1(T^n)$ is surjective. Choosing a preimage $y_i$ of each $x_i$, we have 
\[f^*\left(\bigsqcap_{i=1}^ny_i\right)=\bigsqcap_{i=1}^nf^*(y_i)=\bigsqcap_{i=1}^nx_i\neq 0,\] so $\bigsqcap_{i=1}^n y_i\neq 0$. Thus, the cup length of $X$ is at least $n$, implying the first claim.

For the second claim, let $g$ be a second decomposable $n$-torus with $f$ and $g$ disjoint. As before, we choose a preimage $z_i$ of each $x_i$ under $g^*$. By disjointness (after choosing an inner product on $H_1$, say), we may arrange our choices of $y_i$ and $z_i$ so that $y_i$ restricts to the trivial functional on the image of $g_*$ (resp. $z_i$, $f_*$). The rest of the argument is a little clearer in the case $r=2$, so we treat it first. Setting $\zeta_f:=\bigsqcap_{i=1}^n \zeta(y_i)$ (resp. $g$, $z_i$), we have \begin{align*}
\zeta_f\zeta_g&=\bigsqcap_{i=1}^n(y_i\otimes 1-1\otimes y_i)\bigsqcap_{i=1}^n(z_i\otimes 1-1\otimes z_i)\\
&=\sum_{S,T\subseteq \{1,\ldots, n\}}\pm y_Sz_T\otimes y_{S^c}z_{T^c},
\end{align*} where we have made the abbreviation $y_S=\bigsqcap_{i\in S} y_i$ (resp. $z$, $T$). Since $f^*z_i=0$ and $g^*y_i=0$ for $1\leq i\leq n$, we obtain the equation
\[(f\times g)^*\zeta_f\zeta_g=\bigsqcap_{i=1}^n x_i\otimes \bigsqcap_{i=1}^n x_i\neq0.\] Since $\zeta_f$ and $\zeta_g$ are both $n$-fold products of zero-divisors, the lower bound of $2n$ follows.

For general $r$, given $1\leq a<b\leq r$, we set $\zeta_f^{ab}:=\bigsqcap_{i=1}^n \zeta^{ab}(y_i)$ (resp. $g$, $z_i$) and calculate that
{\small\[\zeta^{12}_f\zeta^{12}_g\bigsqcap_{j=3}^r\zeta_g^{(j-1)j}=\sum_{S_i\subseteq \{1,\ldots, n\}, 1\leq j\leq r}\pm y_{S_1}z_{S_2}\otimes y_{S_1^c}z_{S_2^c}z_{S_3}\otimes z_{S_3^c}z_{S_4}\otimes\cdots\otimes z_{S_{r-1}^c}z_{S_r}\otimes z_{S_r^c}.\]}Applying $(f\times g^{r-1})^*$ and evaluating on the fundamental class of $(T^n)^r$, inspection of the last tensor factor shows that any term with $S_r\neq \varnothing$ vanishes for degree reasons. Moving to the next to last factor, the same reasoning shows that any term with $S_{r-1}\neq \varnothing$ vanishes. Continuing in this way, we are left with the terms satisfying the condition $S_j=\varnothing$ for $j>2$. The same reasoning as in the case $r=2$ permits us to impose the further condition $S_2=S_1^c=\varnothing$, eliminating all but one term. For this last term, we have \[(f\times g^{r-1})^*\left(\bigsqcap_{i=1}^ny_i\otimes \left(\bigsqcap_{i=1}^n z_i\right)^{\otimes r-1}\right)=\left(\bigsqcap_{i=1}^nx_i\right)^{\otimes r},\] which pairs to $1$ with the fundamental class. Thus, the class $\zeta^{12}_f\zeta^{12}_g\bigsqcap_{j=3}^r\zeta_g^{(j-1)j}$ is an $rn$-fold product of $r$-fold zero-divisors, implying the desired lower bound.
\end{proof}

\subsection{Homotopical decomposability} Given the heavy use of (co)homology in the previous section, it may come as something of a surprise that the ideas presented there can be applied in some situations in which an $n$-torus $f:T^n\to X$ induces \emph{zero} on $H_1$---see Section \ref{section:obstacles} for a highly relevant example of this phenomenon.

\begin{definition}
Let $X$ be a topological space.
\begin{enumerate}
\item We say that an $n$-torus in $X$ is \emph{homotopically decomposable} if the homomorphism induced on $\pi_1$ is injective.
\item We say that a pair of $n$-tori in $X$ are \emph{homotopically disjoint} if the conjugates of the images of the two homomorphisms induced on $\pi_1$ intersect trivially.
\end{enumerate}
\end{definition}

For aspherical spaces, there is a direct analoge of Proposition \ref{prop:decomposable}(2) for homotopically disjoint tori, which is a special case of the following result of \cite{FarberOprea:HTCAS} (following \cite{GrantLuptonOprea:NLBTCAS} for $r=2$).

\begin{theorem}\label{thm:FO}
Let $X$ be an aspherical space with $\pi_1(X)=\pi$. Given subgroups $A_i\leq \pi$ for $1\leq i\leq r$ we have \[\tc_r(X)\geq \mathrm{cd}\left(\bigsqcap_{i=1}^r A_i\right)\] provided $\bigcap_{i=1}^rg_iA_ig_i^{-1}=\{1\}$ for every $(g_1, \ldots, g_r)\in \pi^r$.
\end{theorem}

\begin{corollary}\label{cor:homotopical tori}
Let $X$ be an aspherical space. If $X$ admits a homotopically disjoint pair of homotopically decomposable $n$-tori, then $\tc_r(X)\geq rn$ for every $r>1$.
\end{corollary}
\begin{proof}
Let $f$ and $g$ be the tori in question. Define subgroups of $\pi=\pi_1(X)$ by $A_1=\mathrm{im}(f_*)$ and $A_i=\mathrm{im}(g_*)$ for $1<i\leq r$. For any $(g_1,\ldots, g_r)\in \pi^r$, we have 
\[\bigcap_{i=1}^rg_iA_ig_i^{-1}\subseteq g_1A_1g_1^{-1}\cap g_2A_2g_2^{-1}=\{1\}\] by disjointness. Applying Theorem \ref{thm:FO}, and using that $A_i\cong \Z^n$ by decomposability, it follows that 
\[\tc_r(X)\geq \mathrm{cd}\left(\bigsqcap_{i=1}^r A_i\right)=\mathrm{cd}(\Z^{rn})=rn,\]
as desired.
\end{proof}

We close with the simple observation that the homotopical notions are in fact strict weakenings of their homological cousins.

\begin{proposition}
Let $X$ be a topological space. 
\begin{enumerate}
\item If an $n$-torus in $X$ is homologically decomposable, then it is homotopically decomposable.
\item If two homologically decomposable $n$-tori in $X$ are homologically disjoint, then they are homotopically disjoint.
\end{enumerate}
\end{proposition}
\begin{proof}
For the first claim, our assumption is that the bottom arrow is injective in the following commutative diagram of group homomorphisms:
\[\xymatrix{
\pi_1(T^n)\ar[d]_-\cong\ar[r]^-{f_*}&\pi_1(X)\ar[d]\\
H_1(T^n)\ar[r]^-{f_*}&H_1(X).
}\] It follows that either composite is injective, so the top arrow is injective, as desired.

For the second claim, we retain the notation of Corollary \ref{cor:homotopical tori} and suppose for the sake of contradiction that $f$ and $g$ are homologically decomposable, homologically disjoint, and \emph{not} homotopically disjoint. By the last assumption, there exist $a\in A_1$ and $h\in \pi$ such that $hah^{-1}\in A_2$. Since $a\equiv hah^{-1}\,\mathrm{mod} [\pi,\pi]$, the second assumption and the Hurewicz isomorphism imply that $a\in [\pi,\pi]$. Choosing a preimage $\tilde a\in \pi_1(T^n)$ of $a$, we conclude that $\tilde a\mapsto 0$ in $H_1(X)$, contradicting the first assumption.
\end{proof}

\section{Configuration spaces}\label{section:configuration spaces}

The goal of this section is to produce a family of tori in the configuration spaces of a graph $\graf$, to which we aim to apply the methods of the previous section. Our arguments are simple and geometric in nature; essentially, all follows from a single fact about the local behavior of configuration spaces of graphs, to which we now turn.

\subsection{Local flexibility} In this section, we discuss a homotopical principle, tied to the local geometry of a graph, which permits a simplified description of its configuration spaces. This simple but fundamental idea was first articulated in the pioneering work of \cite{Swiatkowski:EHDCSG}, as expanded and clarified in \cite{ChettihLuetgehetmann:HCSTL}. The version we give here, or one very close to it, was discovered by the authors of \cite{AgarwalBanksGadishMiyata:DCCSG} in their search for a purely topological description of certain algebraic constructions of \cite{AnDrummond-ColeKnudsen:SSGBG,AnDrummond-ColeKnudsen:ESHGBG}.

\begin{observation}\label{observation:edge}
On an edge, the exact positions of particles do not matter.
\end{observation}

More precisely, the configuration space of $k$ labeled particles in the unit interval is homeomorphic to a $k!$-fold disjoint union of open $k$-simplices, each of which is contractible.

\begin{observation}\label{observation:vertex}
Near a vertex, only the relative position of the nearest particle matters.
\end{observation}

To see why, imagine lifting the graph by holding this vertex, letting the adjacent edges dangle. Under the influence of gravity, nearby particles gradually slide down these edges until at most one is less than distance $1/2$ from the center. Aside from the position of this nearest particle, the remaining information is a tuple of labeled configurations in a closed interval, to which Observation \ref{observation:edge} applies.

We now give a precise formulation of the flexibility we have been describing. As a matter of notation, given a vertex $v$ of $\graf$, we write $F_k(\graf)_v\subseteq F_k(\graf)$ for the subspace of configurations in which at most one particle lies in the open star of $v$. Given a set of vertices $W$, we write $F_k(\graf)_W=\bigcap_{v\in W} F_k(\graf)_v$ (later we extend this notation in the obvious way to unordered configuration spaces).

\begin{proposition}\label{prop:deformation retract}
Let $W$ a set of vertices in the graph $\graf$. There is a subdivision $\graf'$ of $\graf$ such that the inclusion $F_k(\graf')_W\subseteq F_k(\graf')$ is an equivariant deformation retract for every $k\geq0$.
\end{proposition}

This result is essentially immediate from \cite[Lem. 2.0.1]{AgarwalBanksGadishMiyata:DCCSG} (see also \cite[Prop. 2.9]{Knudsen:OSTCGBG}). We emphasize that the level of subdivision required to invoke Proposition \ref{prop:deformation retract} is independent of $k$.

\begin{remark}
The heuristic of particles moving in $\graf$ under the influence of gravity can be made rigorous using discrete Morse flows \cite{FarleySabalka:DMTGBG}. This approach, premised on the pioneering work of \cite{Abrams:CSBGG}, requires arbitrarily fine subdivision, with the number of vertices growing linearly in $k$, resulting in fearsome combinatorial difficulties. Nevertheless, it has been quite fruitful \cite{FarleySabalka:OCRTBG, FarleySabalka:PGBG,Sabalka:ORIPTBG,KimKoPark:GBGRAAG,KoPark:CGBG,MaciazekSawicki:HGPOCG,Ramos:SPHTBG,AnMaciazek:GPBGPG,GonzalezHoekstra:CTBTBG}.
\end{remark}

\begin{remark}
Taking Observations \ref{observation:edge} and \ref{observation:vertex} to their logical conclusion, \cite{Swiatkowski:EHDCSG} shows that each configuration space of a graph deformation retracts onto a cubical complex in which a cell describes a distribution of particles on edges, together with the information of whether each vertex has a particle approaching it. This complex has proven to be a highly effective conceptual and computational tool \cite{ChettihLuetgehetmann:HCSTL,AnDrummond-ColeKnudsen:ESHGBG,AnDrummond-ColeKnudsen:AHGBG,AnKnudsen:OSHPGBG,ProudfootRamos:CCG}.
\end{remark}

\subsection{Stars and tori}\label{section:stars} This section is devoted to what is both the simplest and the most fundamental example of a graph that is not a manifold, namely the star graph $\stargraph{3}$. In general, we write $\stargraph{n}$ for the cone on the set $\{1,\ldots, n\}$ with its canonical cell structure. Essentially everything we do here will be premised on the following example, whose fundamental importance has long been recognized \cite{Abrams:CSBGG,Ghrist:CSBGGR}.

\begin{example}
Writing $v_0$ for the central vertex of the star graph $\stargraph{3}$, one checks by hand that $F_2(\stargraph{3})_{v_0}$ is a six-pointed ``star of David'' in which each edge represents one particle moving along an edge with the second particle stationary at a univalent vertex. Parametrizing this space piecewise linearly counterclockwise, we obtain the canonical homotopy equivalence $S^1\xrightarrow{\simeq} F_2(\stargraph{3})$ (in this case, there is no need for subdivision). By inspection, this map is $\Sigma_2$-equivariant for the antipodal action on the source, so it descends to a homotopy equivalence $S^1\xrightarrow{\simeq} B_2(\stargraph{3})$.
\end{example}

The cycle of this example describes a situation in which two particles orbit each other by taking turns passing through the central vertex. Performing this move locally at an essential vertex of a larger graph $\graf$, and leaving the remaining $k-2$ particles stationary on some collection of edges, we obtain a cycle in the configuration space of $k$ points in $\graf$. 

\begin{definition}\label{def:star class}
A \emph{star class} in $\graf$ is the image of the fundamental class of $S^1$ under the homomorphism induced by a composite of the form
\[S^1\xrightarrow{\simeq} B_2(\stargraph{3})\cong B_2(\stargraph{3})\times B_1(\stargraph{0})^{k-2}\subseteq B_k(\stargraph{3}\sqcup\stargraph{0}^{\sqcup k-2})\xrightarrow{B_k(\varphi)}B_k(\graf)\] where $\varphi$ is a topological embedding.
\end{definition}

Iterating this construction over multiple vertices, we obtain a rich family of toric homology classes.

\begin{definition}\label{def:W-torus}
Let $W$ be a set of essential vertices. A $W$-\emph{torus} in $\graf$ is the image of the fundamental class of $(S^1)^W$ under the homomorphism induced by a composite of the form
{\small\[
(S^1)^W\xrightarrow{\simeq} B_2(\stargraph{3})^W\cong B_2(\stargraph{3})^W\times B_1(\stargraph{0})^{k-2|W|}\subseteq B_k(\stargraph{3}^{\sqcup W}\sqcup\stargraph{0}^{\sqcup k-2|W|})\xrightarrow{B_k(\varphi)}B_k(\graf)\]}where $\varphi$ is a topological embedding. If $W$ is the set of all essential vertices, we say the torus is \emph{maximal}.
\end{definition}

These constructions apply equally to ordered configuration spaces, the only difference being that we must specify \emph{which} pairs of particles orbit one another at which vertices. Our terminology will mostly fail to distinguish between the ordered and unordered cases, and most statements are valid for both. In particular, the following fundamental result holds in both settings.

\begin{proposition}\label{prop:torus nonzero}
A maximal torus has infinite order.
\end{proposition}
\begin{proof}
It suffices to prove the claim for unordered tori, since the image of an ordered $W$-torus under the homomorphism induced by the projection $F_k(\graf)\to B_k(\graf)$ differs from the corresponding unordered $W$-torus by a factor of $2^{|W|}$. Choosing a maximal torus and setting $m=m(\graf)$, we assume for simplicity that the $k-2m$ stationary particles are located on a single edge (it does not affect the argument). We may also assume after subdividing that the conclusion of Proposition \ref{prop:deformation retract} holds. Writing $\graf_v$ for the complement of the open star of $v$, we have the cofiber sequence
\[\xymatrix{
B_k(\graf_v)\ar[r]& B_k(\graf)_v\ar[r]& B_{k-1}(\graf_v)_+\wedge\stargraph{d(v)}/\partial \stargraph{d(v),}
}\]where the righthand map is defined by recording separately the positions of particles lying inside and outside of the open star at $v$. Iterating this construction over the set $W$ of essential vertices, appealing to Proposition \ref{prop:deformation retract}, applying homology, invoking the K\"{u}nneth isomorphism, and choosing an ordering of the edges incident on each $v\in W$, we obtain a homomorphism
\[\xymatrix{
H_{m}(B_k(\graf))\ar[r]& H_0(B_{k-m}(\graf\setminus W))\otimes \displaystyle\left(\bigotimes_{v\in W} H_1(\stargraph{d(v)}/\partial\stargraph{d(v)})\right).
}\] The target being free Abelian, it suffices to show that our maximal torus does not lie in the kernel of this homomorphism.

To attack this question, we choose the basis $H_1(\stargraph{d(v)}/\partial\stargraph{d(v)})\cong\Z\langle h^v_{2},\ldots, h^v_{d(v)}\rangle$, where $h^v_i$ is represented by the loop that travels from the basepoint to the central vertex along the $1$st edge and back to the basepoint along the $i$th edge. By Observation \ref{observation:edge} and simple counting, we further recognize $H_0(B_{k-m}(\graf\setminus W))$ as the degree $k-m$ summand of the polynomial ring $\Z[\pi_0(\graf\setminus W)]$. Thus, we may regard the image of the homomorphism in question as lying in the free $\Z[\pi_0(\graf\setminus W)]$-module generated by the set $\bigsqcap_{v\in W}\{h_i^v\}_{i=2}^{d(v)}$. 

We may assume without loss of generality that our ordering of the edges at each vertex begins with the three edges distinguished by the embedding of $\stargraph{3}$ at that vertex, in order. One then checks directly---by reducing to the case $\graf=\stargraph{3}$, for example---that the image of our maximal torus under the above homomorphism is represented in this basis by the expression
\[e_0^{k-2m}\bigsqcap_{v\in W}\left((e_3^v-e_1^v)h_{2}^v+(e_2^v-e_1^v)h_{3}^v)\right),\] where $e_i^v\in \pi_0(\graf\setminus W)$ is the component of the $i$th edge at $v$ and $e_0$ is the component of the edge containing the $k-2m$ stationary particles. Perhaps after relabeling, we may assume that $e_1^v\neq e_2^v$ for every essential vertex $v\in W$---here, we use our assumption that $W$ is the set of \emph{all} essential vertices---in which case the monomial $e_0^{k-2m}\bigsqcap_{v\in W}(e_2^v-e_1^v)$ is nonzero, since the polynomial ring is an integral domain. Since this monomial is the coefficient of the basis element $\bigsqcap_{v\in W} h^v_3$ in the above expression, the claim follows.
\end{proof}

\begin{remark}
The key property used in this argument is that, at each vertex in $W$, two edges involved in the $W$-torus lie in different components of $\graf\setminus W$. This observation leads to the concept of \emph{rigidity} and a stronger version of Proposition \ref{prop:torus nonzero} for rigid tori---see \cite{AnDrummond-ColeKnudsen:AHGBG}, which proves that, asymptotically in $k$, all homological phenomena are dominated by rigid tori, at least in the unordered case.
\end{remark}

\begin{remark}
The cofiber sequence used in the proof of Proposition \ref{prop:torus nonzero} gives rise to a long exact sequence in homology, known as the \emph{vertex explosion exact sequence} \cite{AnDrummond-ColeKnudsen:ESHGBG}. This exact sequence is arguably the single most powerful tool in studying the homology of configuration spaces of graphs.
\end{remark}

\begin{remark}
The appearance of the polynomial ring above is not as fanciful as it may seem; indeed, the homology of the configuration spaces of $\graf$ forms a bigraded module over the polynomial ring generated by the edges of $\graf$ \cite{AnDrummond-ColeKnudsen:ESHGBG}.
\end{remark}

\subsection{Universal upper bounds} In this section, we use local flexibility to give easy derivations of the fundamental results of {\'{S}}wi\k{a}tkowski \cite{Swiatkowski:EHDCSG}.\footnote{Strictly speaking, {\'{S}}wi\k{a}tkowski studies the a priori stronger homotopical dimension---see also \cite{Ghrist:CSBGGR,FarleySabalka:DMTGBG}.}

\begin{theorem}\label{thm:upper bound}
For any connected graph $\graf$ and $k\geq0$, we have $\cat(B_k(\graf))\leq m(\graf)$.
\end{theorem}
\begin{proof}
Writing $W$ for the set of essential vertices of $\graf$, we may assume after subdivision that $\graf$ satisfies the hypotheses of Proposition \ref{prop:deformation retract}. Given $0\leq p\leq m(\graf)$, write $X_p\subseteq B_k(\graf)_W$ for the subspace of configurations in which at most $p$ particles lie in the open star of $W$. It suffices to show that $\cat(X_p)\leq p$ for $p>0$; indeed, the case $p=m(\graf)$ is equivalent to the claim by homotopy invariance and Proposition \ref{prop:deformation retract}.

We proceed by induction on $p$. For the induction step, write $Y_p\subseteq X_p$ for the subspace of configurations in which no more than $p-1$ particles are located at members of $W$. The subspaces $Y_p$ and $X_p\setminus X_{p-1}$ are open and together cover $X_p$. Using the pasting lemma indexed over subsets of $W$ of cardinality $p-1$, it is easy to show that $Y_p$ deformation retracts onto $X_{p-1}$, while $X_p\setminus X_{p-1}$ is homotopy equivalent to a disjoint union of products of configuration spaces of intervals. Since the latter are all contractible, it follows that the inclusion $X_p\setminus X_{p-1}\subseteq X_p$ is nullhomotopic, so $\cat(X_p)=\cat(Y_p)+1=p$ by induction.

For the base case $p=1$, we need only note that $X_0$ is now also a disjoint union of products of configuration spaces of intervals, so $\{X_1\setminus X_0, Y_1\}$ is a categorical open cover.
\end{proof}

\begin{remark}
Generalizing an observation made in the proof of Proposition \ref{prop:torus nonzero}, we may identify the associated graded object for this filtration as
\[X_p/X_{p-1}\cong B_{k-p}(\graf\setminus W)_+\wedge \bigvee_{S\subseteq W: |S|=p}\bigwedge_{v\in S} \stargraph{d(v)}/\partial\stargraph{d(v)},\] which has the homotopy type of a bouquet of $p$-spheres by Observation \ref{observation:edge}. Thus, as in the case of the skeletal filtration of a CW complex, the resulting spectral sequence in homology is simply a chain complex. This comparison is more than an analogy; indeed, the skeletal filtration of {\'{S}}wi\k{a}tkowski's cubical complex is homotopy equivalent to the filtration considered here.
\end{remark}

In order to proceed, we require what is perhaps the most fundamental result on configuration spaces of graphs.

\begin{theorem}\label{thm:asphericity}
Configuration spaces of graphs are aspherical.
\end{theorem}

\begin{remark}
More precisely, \cite{Abrams:CSBGG} shows that the fundamental groups of configuration spaces of graphs are graphs of groups, and both \cite{Abrams:CSBGG} and \cite{Swiatkowski:EHDCSG} show that the configuration spaces themselves are homotopy equivalent to CAT(0) cube complexes. Either result implies asphericity.
\end{remark}

\begin{remark}
Theorem \ref{thm:asphericity} is essentially the only result about configuration spaces of graphs that the author does not yet know how to deduce from Proposition \ref{prop:deformation retract}. He invites the reader to do so.
\end{remark}

\begin{corollary}\label{cor:cat}
For any $k\geq 2m(\graf)$, we have $\cat(B_k(\graf))=\cat(F_k(\graf))=m(\graf)$.
\end{corollary}
\begin{proof}
By the Eilenberg--Ganea theorem \cite{EilenbergGanea:LSCAG} and Theorem \ref{thm:asphericity}, the quantity $\cat(B_k(\graf))$ coincides with its cohomological dimension. In the range $k\geq 2m(\graf)$, Proposition \ref{prop:torus nonzero} implies that $H_{m(\graf)}(B_k(\graf);\Q)\neq 0$, so $m(\graf)$ is a lower and, by Theorem \ref{thm:upper bound}, also an upper bound. The argument in the ordered case is identical.
\end{proof}

\begin{corollary}\label{cor:upper bound}
For any connected graph $\graf$, $k\geq0$, and $r>0$, the quantities $\tc_r(F_k(\graf))$ and $\tc_r(B_k(\graf))$ are bounded above by $r m(\graf)$.
\end{corollary}
\begin{proof}
By the same asphericity considerations as in the previous argument, the quantity $\tc_r(F_k(\graf))$ is bounded above by the cohomological dimension of the Cartesian power $B_k(\graf)^r$. In light of the rational K\"{u}nneth isomorphism, Proposition \ref{prop:torus nonzero} yields the claim as before.
\end{proof}

\section{Detection}\label{section:detection}

In light of the strategies for bounding topological complexity discussed in Section \ref{section:tori}, there are two obvious questions. Are the maximal tori constructed in Section \ref{section:configuration spaces} decomposable? And, if so, can we find a disjoint pair of such? As we will see, the answer to the first question is ``{sometimes}'' homologically and ``{always}'' homotopically---it is precisely the equivocal nature of this answer that blocked progress on Farber's conjecture for so long. The answer to the second question is ``yes'' in the ordered setting; in fact, due to the action of the symmetric group, finding two disjoint decomposable tori is no harder than finding the first. In the unordered setting, the best known answer is ``with assumptions on the graph, eventually.'' Completing this answer is probably the most interesting open problem in the area---see Section \ref{section:problems}.

In every case, the approach to establishing decomposability is to \emph{detect} the image (in $H_1$ or in $\pi_1$) of the torus, a map \emph{into} the configuration space, with an appropriate map \emph{out}. To establish disjointness, we need only find a detector for the first map that fails to detect the second.

\subsection{Homological strategies}\label{section:homological strategies} In the ordered case, the crux of the matter of detection is simple. Consider the commutative diagram
\[\xymatrix{
(S^1)^W\ar[d]\ar[rr]&&F_k(\graf)\ar[d]\\
S^1\ar[r]^-\simeq&F_2(\stargraph{3})\ar[r]^-{F_2(\varphi_v)}&F_2(\graf),
}\] where the top arrow is a maximal torus, the lefthand map is projection onto the factor indexed by the essential vertex $v$, the righthand map is the projection onto the coordinates indexing the two particles located near $v$, and $\varphi_v$ is a topological embedding sending $v_0\in \stargraph{3}$ to $v$. If it is possible to choose the embeddings $\varphi_v$ in such a way that the bottom composite induces an injection on $H_1$ for every $v\in W$, then our maximal torus is detectable via the map $F_k(\graf)\to \bigsqcap_{v\in W} F_2(\graf)$ given by the appropriate coordinate projections. 

With this detection strategy, disjointness is automatic. Restricting (without loss of generality) to the case $k=2m(\graf)$, and numbering the essential vertices of $\graf$, consider the $W$-torus in which particles $1$ and $2$ orbit vertex $1$, particles $3$ and $4$ orbit vertex $2$, and so on. This torus is disjoint from the one obtained by applying a cyclic permutation to the labels of the particles; indeed, the homomorphism induced by the detection map for the first vanishes on the image (in $H_1$ or in $\pi_1$) of the second.

\begin{remark}
Although we will have no cause to use it directly, no discussion of the first homology of configuration spaces of graphs would be complete without reference to Ko--Park's pioneering work \cite{KoPark:CGBG}. The author finds it likely that every fact mentioned below concerning these groups is implicit---though likely not obvious---in that work.
\end{remark}

\subsubsection{Trees} The detection strategy described above was first implemented by Farber in his proof of the $r=2$ case of Theorem \ref{thm:ordered} for trees \cite[Thm. 10]{Farber:CFMPG}. In this case, one has access to a deformation more radical than those discussed thus far \cite[Thm. 12]{Farber:CFMPG}.

\begin{theorem}\label{thm:tree}
For a tree $\graphfont{T}$, the space $F_2(\graphfont{T})$ has the homotopy type of a graph.
\end{theorem}

A cycle in this graph corresponds to one of the star cycles described in Section \ref{section:stars}; in other words, in this case, the bottom map in the above commutative diagram is, up to homotopy, the inclusion of a cycle in a graph, hence certainly injective on homology. 

\subsubsection{Articulations}\label{subsection:articulation} The success of our detection strategy for trees prompts an obvious general question: when is the homomorphism $H_1(F_2(\stargraph{3}))\to H_1(F_2(\graf))$ induced by a topological embedding injective? A partial answer to this question was given in \cite{LuetgehetmannRecio-Mitter:TCCSFAGBG}.

\begin{proposition}\label{prop:injective articulation}
Let $\varphi:\stargraph{3}\to \graf$ be a topological embedding with $\varphi(v_0)=v$. If $\mathrm{im}(\varphi)$ intersects more than one component of $\graf\setminus \{v\}$, then $F_2(\varphi)_*$ is injective.
\end{proposition}

In order to apply this proposition, it must be the case that removing the vertex in question causes the graph to become disconnected. Such a vertex is sometimes called an \emph{articulation}; thus, calling graphs in which every essential vertex is an articulation ``fully articulated,'' our detection strategy immediately implies that a fully articulated graph admits a pair of disjoint decomposable maximal tori, establishing Theorem \ref{thm:ordered} for such graphs \cite[Thm. A]{LuetgehetmannRecio-Mitter:TCCSFAGBG}

We encourage the reader to prove Proposition \ref{prop:injective articulation} herself by adapting the argument of Proposition \ref{prop:torus nonzero}; in fact, the argument in this case is much simpler. The proof given in \emph{loc. cit.} is different, relying instead on the notion of a graph with \emph{sinks}, which are special vertices at which particles are permitted to collide. We will return to these ideas below in Section \ref{section:homotopical strategies}, where they will be a key component in the proofs of the main results.

\begin{remark}
The sufficient condition of Proposition \ref{prop:injective articulation} is far from necessary---see Section \ref{section:planar} below, for example.
\end{remark}

\subsubsection{Discrete Morse methods} The problem of homological detection can equivalently be viewed as a problem of producing cohomology classes with certain desirable properties. Surely, then, one should be able to approach the problem by studying the cohomology ring of the configuration spaces directly. Unfortunately, these rings are very complicated and almost completely unknown in general, the notable exception being the case of a tree. Here, the differential in Farley--Sabalka's discrete Morse complex vanishes \cite{FarleySabalka:OCRTBG}, reducing the problem to combinatorics (albeit combinatorics of substantial complexity unbounded in $k$). 

This simplification has been successfully exploited in studying the detection problem and deducing results of various flavors for (sequential) topological complexity, in particular a proof of Theorem \ref{thm:ordered} for trees \cite{Aguilar-GuzmanGonzalezHoekstra-Mendoza:FSMTMHTCOCST}. Notably, these techniques sometimes apply outside of the constraints of Farber's original conjecture, namely to unordered configurations \cite{Scheirer:TCUCSCG,HoekstraMendoza:BHTCCST} and small numbers of particles.

\subsubsection{Planar graphs}\label{section:planar} In view of its complexity and mystery, one wants to deal directly with the cohomology of $F_k(\graf)$ as little as possible. The main insight of \cite{Knudsen:FCPG} is that, for planar graphs, one can circumvent it entirely, working instead with the cohomology of the configuration spaces of the plane, which are exceedingly well studied \cite{Arnold:CRCBG}. In our current setup, all that is required is the following easy observation, which the reader is invited to prove.

\begin{proposition}\label{prop:planar star}
Any piecewise linear embedding $\stargraph{3}\to \R^2$ induces a homotopy equivalence $F_2(\stargraph{3})\simeq F_2(\R^2)$.
\end{proposition}

It follows that \emph{any} embedding of $\stargraph{3}$ into a planar graph $\graf$ induces an injection on $H_1(F_2(-))$. Indeed, choosing a piecewise linear embedding of $\graf$ into $\R^2$ and assuming without loss of generality that the embedding of $\stargraph{3}$ is also piecewise linear, the composite map \[F_2(\stargraph{3})\to F_2(\graf)\to F_2(\R^2)\] is a homotopy equivalence, so the first map is injective on homology. This reasoning establishes Theorem \ref{thm:ordered} for all planar graphs \cite[Thm. 1.1]{Knudsen:FCPG}.

\subsection{Obstacles}\label{section:obstacles} Thus far, we have said nothing about the unordered case. The outlook in this setting is rather grim, as our strategy for establishing disjointness relied crucially on the action of the symmetric group. The situation is even worse than this, however; in fact, we will be hard pressed even to establish decomposability! As it turns out, the obstacle here and the obstacle to extending further in the ordered setting arise from a common root.

We write $\thetagraph{n}$ for the suspension of the set $\{1,\ldots, n\}$ with its minimal cell structure---see Figure \ref{figure:theta}.

\begin{proposition}\label{prop:theta}
In $H_1(B_2(\thetagraph{3}))$, the star classes at the bottom and top vertices, both oriented counterclockwise, coincide.
\end{proposition}
\begin{proof}
Notice first that both classes are invariant under the action of the cyclic group $C_3$ on the edges of $\thetagraph{3}$. On the other hand, choosing one of the two vertices and subdividing appropriately, we have the $C_3$-equivariant exact sequence
\[
\xymatrix{
H_1(B_2(\stargraph{3}))\ar[r]& H_1(B_2(\thetagraph{3}))\ar[r]&H_1(\stargraph{3}/\partial\stargraph{3})
}\] as in the proof of Proposition \ref{prop:torus nonzero}, and the rightmost group in this sequence is isomorphic as a $C_3$-module to the reduced regular representation, which has no fixed points. Since the lefthand map is the inclusion of the star class at the other vertex, it follows that the two classes are linearly dependent. Embedding $\thetagraph{3}$ in the plane, it is easy to see that a generator of $H^1(F_2(\R^2))$ takes the same value on both classes, implying the claim.
\end{proof}

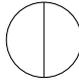
\begin{figure}[ht]
\begin{tikzpicture}
\begin{scope}
\draw(0,0) circle (.5cm);
\draw(0,.5) -- (0,-.5);
\end{scope}

\end{tikzpicture}
\caption{The theta graph $\thetagraph{3}$.}\label{figure:theta}
\end{figure}

This ``$\theta$-relation'' implies that, even if each individual circle factor in a torus is detectable in homology, the various images could overlap, destroying decomposability. We immediately conclude, for example, that $\thetagraph{3}$ admits no indecomposable unordered maximal torus, destroying any hope of adapting our strategy to the unordered setting without drastic modification.

Unfortunately, the same relation (or rather its double cover) holds in the ordered setting, where it destroys any hope of adapting our strategy beyond the limits of articulations and planarity detailed above. 

\begin{proposition}\label{prop:planarity}
The graph $\graf$ is non-planar if and only if some ordered maximal torus in $\graf$ is indecomposable.
\end{proposition}
\begin{proof}
Consider the complete bipartite graph $\completegraph{3,3}$ as depicted in Figure \ref{figure:nonplanar}. This graph receives a piecewise linear embedding from $\thetagraph{3}$ taking the top vertex to the upper left vertex and the bottom to the middle left. Applying the $\theta$-relation, it follows that the counterclockwise oriented star class at upper left is equal to the clockwise oriented star class at middle left. Repeating this argument with the other pairs of lefthand vertices, we conclude that the star class at upper left is equal to its own additive inverse. By symmetry, \emph{any} star class in $\completegraph{3,3}$ has order at most $2$, and the same holds for the complete graph $\completegraph{5}$ by a similar argument. Since any non-planar graph admits a topological embedding from one of these two graphs by Kuratowski's theorem, the ``only if'' claim follows. The ``if'' claim follows from Proposition \ref{prop:planar star}.
\end{proof}

\begin{figure}[ht]
\centering
\begin{tikzpicture}
\draw (-.75,.75) -- (.75,.75) -- (-.75,0) -- (.75,0) -- (-.75, -.75) -- (.75, .75);
\draw (.75, -.75) -- (-.75,.75) -- (.75, 0);
\draw (-.75,0) -- (.75,-.75)-- (-.75,-.75);
\begin{scope}[xshift= 4cm]
\node[draw=none,minimum size=2.5cm,regular polygon,regular polygon sides=5] (a) {};
\foreach \x in {1,...,5}
\foreach \y in {1,...,\x}
\draw(a.corner \x) -- (a.corner \y);
\draw(-.3,1.3) node {};
\draw(.3,1.3) node {};
\draw(0,1.45) node{};
\draw(-1.13,.7) node {};
\draw(-1.27,.1) node {};
\draw(-1.35,.45) node{};
\end{scope}

\end{tikzpicture}
\caption{The atomic non-planar graphs $\completegraph{3,3}$ and $\completegraph{5}$.}\label{figure:nonplanar}
\end{figure}
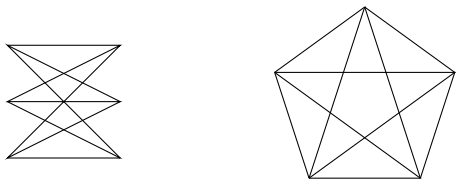

\begin{remark}
The star classes shown to have order at most $2$ are in fact trivial, since the homology group in question is torsion-free \cite{KoPark:CGBG}. In the unordered case, the same argument establishes the ``only if'' claim, and the star classes in this case are nontrivial, as indeed every unordered star class is \cite{AnDrummond-ColeKnudsen:SSGBG}. This example represents some of the only known torsion in the homology of configuration spaces of graphs.
\end{remark}

Nevertheless, there is reason for hope. Although decomposability may fail for some or even all of our tori, this pathology is an artifact of working with homology; indeed, at the level of fundamental groups of configuration spaces, \emph{any} embedding between \emph{any} two graphs induces an injection.\footnote{We will not use this result, but see \cite{Swiatkowski:EHDCSG} for a proof premised on non-positive curvature. A purely topological proof is also available by combining results of \cite{AnPark:SBGC}.}

\subsection{Homotopical strategies}\label{section:homotopical strategies} The key to finding our way forward is locality. The content of the $\theta$-relation discussed above is that homological information in the configuration spaces of a graph is not localized at essential vertices; given sufficient topological connectivity, it flows freely throughout the graph. To overcome this obstacle, therefore, we will focus our attention solely on local information. In this effort, we borrow an idea first introduced in \cite{ChettihLuetgehetmann:HCSTL} and further studied in \cite{Ramos:CSGCTPC,LuetgehetmannRecio-Mitter:TCCSFAGBG}.

\begin{definition}
A \emph{graph with sinks} is a pair $(\graf,S)$ of a graph $\graf$ and a collection $S$ of vertices, called \emph{sinks}.
\end{definition}

A graph with sinks $(\graf,S)$ has ordered and unordered configuration spaces, properly denoted $F_k(\graf,S)$ and $B_k(\graf,S)$ respectively, in which particles are required to be distinct away from sinks. Henceforth, we suppress the set of sinks from the notation, as it will always be clear from context.

Graphs with sinks arise naturally as quotients of graphs by collections of subgraphs, and the quotient map in such a situation induces a map at the level of configuration spaces. The local information referenced above will be captured by the quotient of $\graf$ by the complement of the open star of a vertex. Thus, the examples of interest for us will all be quotients of star graphs by collections of univalent vertices.

\begin{definition}\label{def:local graph}
Let $I$ be a finite set and $\pi$ an equivalence relation on $I$. The \emph{local graph} on $\pi$ is the graph $\graphfont{\Lambda}(\pi)$ with set of vertices $\{v_0\}\sqcup\{v_{[i]}\}_{i\in I/\pi}$, in which $v_0$ and $v_{[i]}$ share a set of edges indexed by the equivalence class $[i]$. We regard every vertex of $\graphfont{\Lambda}(\pi)$ as a sink except for $v_0$.
\end{definition}

See Figure \ref{figure:local graphs} for two pertinent examples of this construction.

\begin{example}
We write $\stargraph{n}/\partial$ for the local graph on the indiscrete equivalence relation on $\{1,\ldots, n\}$. As a graph, $\stargraph{n}/\partial$ is the suspension of $n$ points, with one sink and one non-sink vertex.
\end{example}

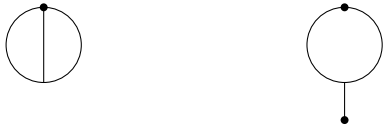
\begin{figure}[ht]
\begin{tikzpicture}
\begin{scope}
\fill[black] (0,.5) circle (1.5pt);
\draw(0,0) circle (.5cm);
\draw(0,.5) -- (0,-.5);
\end{scope}
\begin{scope}[xshift=4cm]
\fill[black] (0,.5) circle (1.5pt);
\fill[black] (0,-1) circle (1.5pt);
\draw(0,0) circle (.5cm);
\draw(0,-1) -- (0,-.5);
\end{scope}


\end{tikzpicture}
\caption{The local graphs $\stargraph{3}/\partial$ and $\graphfont{\Lambda}(\pi_{23})$ where $\pi_{23}$ is the equivalence relation $2\sim 3$ on $\{1,2,3\}$.}\label{figure:local graphs}
\end{figure}

These local graphs have the wonderful property that their configuration spaces are also graphs (compare Theorem \ref{thm:tree}). A precise statement requires a certain amount of notational overhead, but the idea is extremely simple.

\begin{construction}
Fixing an equivalence relation $\pi$ on $I$ and $k\geq0$, let $\graphfont{\Lambda}_k(\pi)$ be the graph specified as follows.
\begin{enumerate}
\item For every partition $p=(p_{[i]})_{[i]\in I/\pi}$ of $k$ or of $k-1$, there is a vertex $v_p$.
\item\label{edge description} If $p'$ is obtained from $p$ by replacing $p_{[i]}$ by $p_{[i]}-1$, then $v_p$ and $v_{p'}$ share a set of edges indexed by the equivalence class $[i]$.
\end{enumerate}
We define a function $f_k:{\graphfont{\Lambda}}_k(\pi)\to B_k(\graphfont{\Lambda}(\pi))$ as follows.
\begin{enumerate}[resume]
\item The configuration $f_{k}(v_p)$ comprises $p_{[i]}$ particles located at $v_{[i]}$. If $p$ is a partition of $k-1$, then $f_{k}(v_p)$ also comprises a particle located at $v_0$.
\item If $p$ and $p'$ are as in (2), then the image under $f_{k}$ of the edge indexed by $j\in[i]$ is the path in which the particle at $v_0$ moves linearly onto $v_{[i]}$ along the edge of $\graphfont{\Lambda}(\pi)$ indexed by $j$.
\end{enumerate}
\end{construction}

\begin{proposition}\label{prop:local graph}
The function $f_k$ is a homotopy equivalence.
\end{proposition}
\begin{proof}
In essence, extending notation in the obvious way, it is easy to show that  $f_k$ is a homeomorphism onto $B_k(\graphfont{\Lambda}(\pi))_{v_0}$, so the claim follows from a version of Proposition \ref{prop:deformation retract} for graphs with sinks---see \cite{Knudsen:OSTCGBG}.
\end{proof}

Now, in contrast to the situation with trees, it often happens that a star class in $\graf$ becomes trivial in the corresponding local graph; fortunately, we may overcome this difficulty by working with the fundamental group.

\begin{example}\label{example:indiscrete}
If $\pi$ is the indiscrete equivalence relation on $\{1,\ldots, n\}$, then $\graphfont{\Lambda}_2(\pi)\cong\thetagraph{n}$, so $B_2(\stargraph{n}/\partial)$ is homotopy equivalent to a wedge sum of $n-1$ circles. It is easy to see that the fundamental group is freely generated by the set $\{[\gamma_1],\ldots, [\gamma_{n-1}]\}$, where $\gamma_i$ is the loop in which a single particle moves from the sink vertex to the non-sink vertex along the $i$th edge and back along the $(i+1)$st edge. With this in mind, it is easy to check that the quotient map $\stargraph{3}\to \stargraph{3}/\partial$ sends the generator of $\pi_1(B_2(\stargraph{3}))=\Z$ to the commutator $[[\gamma_1],[\gamma_2]]$.
\end{example}

\begin{example}\label{example:23}
Writing $\pi_{23}$ for the equivalence relation $2\sim 3$ on $\{1,2,3\}$, the graph $\graphfont{\Lambda}_3(\pi_{23})$ is depicted in Figure \ref{figure:three points}; thus, $\pi_1(B_3(\graphfont{\Lambda}(\pi_{23})))$ is free on $3$ generators. Consider the loop of configurations in $\stargraph{3}$ in which two particles orbit each other by passing through the central vertex, while a third particle remains stationary on the first edge. Under the homomorphism on fundamental groups induced by the map $B_3(\stargraph{3})\to B_3(\graphfont{\Lambda}(\pi_{23}))$, this loop is sent to the product of the first and second generators. In contrast, the loop in which the third particle remains stationary on the second (or third) edge is sent to the product of the second and third generators.
\end{example}

\begin{figure}[ht]
\begin{tikzpicture}[rotate=180]
\begin{scope}
\draw(0,0) circle (.5cm);
\draw(2,0) circle (.5cm);
\draw(4,0) circle (.5cm);
\draw(.5,0) -- (1.5,0);
\draw(2.5,0) -- (3.5,0);
\draw(4.5,0) -- (5.5,0);
\end{scope}


\end{tikzpicture}
\caption{The graph $\graphfont{\Lambda}_3(\pi_{23})$.}\label{figure:three points}
\end{figure}
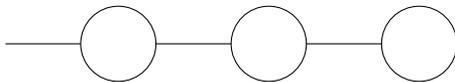

Surprisingly, these two simple examples are enough to yield the main results. In the ordered setting, we may take $k=2m(\graf)$ without loss of generality. As in Section \ref{section:homological strategies}, we obtain two maps $(S^1)^{m(\graf)}\to F_{k}(\graf)$ by choosing an embedding of $\stargraph{3}$ at each essential vertex and varying which particles orbit which vertices using a cyclic permutation. By Example \ref{example:indiscrete}, symmetry, and naturality, these two $m(\graf)$-tori are homotopically decomposable and homotopically disjoint, as witnessed by maps of the form
\[F_k(\graf)\to \bigsqcap_{v\in W}F_2(\graf)\to \bigsqcap_{v\in W}B_2(\graf)\to \bigsqcap_{v\in W} B_2(\stargraph{d(v)}/\partial)\] where the first is given by the relevant coordinate projections, the second forgets orderings in each factor, and the third is induced by collapsing the complement of the open star of a vertex to a point in each factor. The missing ingredient to connect to Example \ref{example:indiscrete} is the fact that $B_2(\stargraph{3}/\partial)\subseteq B_2(\stargraph{n}/\partial)$ is a homotopy retract, since it is homotopic to the inclusion of a connected subgraph by Proposition \ref{prop:local graph}.

In the unordered setting, we proceed as follows. At each vertex $v$ with $d(v)>3$, we choose two embeddings of $\stargraph{3}$ involving distinct sets of edges. At each trivalent vertex, all of which are separating by assumption, we choose two edges lying in different components of the complement. We obtain an $m(\graf)$-torus in $B_k(\graf)$ for any $k$ at least $2m(\graf)$ plus the number of trivalent vertices by allowing pairs of particles to orbit the essential vertices according to our chosen embeddings, with stationary particles near the trivalent vertices on our chosen edges. Invoking both examples, now, these two $m(\graf)$-tori are likewise homotopically decomposable and homotopically disjoint, as witnessed by the map
\[B_k(\graf)\to \bigsqcap_{v\in W}B_k(\graphfont{\Lambda}(\pi_v)\] given by the quotient map in each factor, where $\pi_v$ is the equivalence relation on the set of edges at $v$ given by the connected components of $\graf\setminus v$. If $d(v)>3$, then we quotient further to $\stargraph{d(v)}/\partial$ and connect to Example \ref{example:indiscrete} via a homotopy retraction onto $B_2(\stargraph{d(v)}/\partial)$. If $d(v)=3$, then we quotient further to $\graphfont{\Lambda}(\pi_{23})$, where edges $1$ and $2$ lie in distinct components of $\graf\setminus \{v\}$, and connect to Example \ref{example:23} via a homotopy retraction onto $B_3(\graphfont{\Lambda}(\pi_{23}))$.

Regardless, with a little care around issues of basepoints, these considerations and Corollary \ref{cor:homotopical tori} give a lower bound of $rm(\graf)$, which yields the claims when combined with Corollary \ref{cor:upper bound}.

\section{Open problems}\label{section:problems}

We close by highlighting some particularly interesting directions for future inquiry. We organize our thoughts around a longstanding orienting principle in the study of configuration spaces, namely the dichotomy of stability and instability. Writing $k$ for the number of particles, as above, a phenomenon is said to be \emph{stable} if it occurs for all sufficiently large $k$, or perhaps in the limit as $k$ tends to infinity (see \cite{McDuff:CSPNP,ChurchEllenbergFarb:FIMSRSG}, for example). Stable phenomena tend to be structured and calculable, unstable phenomena ephemeral and subtle. In this sense, Farber's conjecture is a stable problem, and framing it as such motivates extracting three separate questions (each with its obvious unordered counterpart).
\begin{itemize}
\item {\bf Stability.} Is the sequence $\{\tc_r(F_k(\graf))\}_{k\geq k_0}$ constant for some $k_0$?
\item {\bf Stable value.} What is $\tc_r(F_{k_0}(\graf))$?
\item {\bf Stable range.} What is the minimal value of $k_0$?
\end{itemize}
To the first question, we have seen that stability holds in general in the ordered setting, and we have established it in the unordered setting under certain assumptions on the topology of the graph.

\begin{problem}
Does stability hold in the unordered setting for graphs with non-separating trivalent vertices? 
\end{problem}

Even if this problem has a positive answer, it is not guaranteed that the stable topological complexity will be maximal.

\begin{problem}
Calculate $\lim_{k\to\infty} TC_r(B_k(\graf))$ in terms of the number of non-separating trivalent vertices of $\graf$.
\end{problem}

The simplest example of a graph with a non-separating trivalent vertex, namely the theta graph $\thetagraph{3}$, is probably also the most fundamental. Indeed, it seems plausible that the local techniques discussed above would serve to reduce these problems to the following.

\begin{problem}
Calculate $\tc(B_k(\thetagraph{3}))$ for large $k$.
\end{problem}

As for the stable range problem, our understanding is very limited. While the bound given by Farber's conjecture is sharp in general, examples are known where $k_0=2$ and $m(\graf)$ is arbitrarily large \cite{GonzalezGonzalez:ADGFCACSTPCG}. It seems that the nature of the stable range depends heavily on the topology of the background graph; in particular, the connectivity and first Betti number appear to be relevant.

\begin{problem}
Calculate the stable range (ordered or unordered) in terms of graph invariants.
\end{problem}

As with many stable problems, the corresponding unstable problem is interesting and much less well understood.

\begin{problem}
Calculate $\tc_r(B_k(\graf))$ and $\tc_r(F_k(\graf))$ for small $k$ in terms of graph invariants. 
\end{problem}

A complete answer is known for trees \cite{Aguilar-GuzmanGonzalezHoekstra-Mendoza:FSMTMHTCOCST}, but little is known for graphs with positive first Betti number. Remarkably, this problem is already interesting in the case $r=1$ of the Lusternik--Schnirelmann category!

It would likely be helpful in tackling these problems to know whether $\tc_r(B_k(\graf))$ is monotone in $k$, a property that holds in the ordered setting.

\begin{problem}
Is the function $k\mapsto \tc_r(B_k(\graf))$ non-decreasing? What if $\graf$ is replaced by a general topological space?
\end{problem}

As this story develops, it is likely that graphs with sinks will continue to play an important role, and they may be interesting objects of study in their own right.

\begin{problem}
Does an analogue of Farber's conjecture hold for configuration spaces of graphs with sinks?
\end{problem}

In the classical setting, as we have seen, (co)homology only gets us so far, past which point asphericity is key. Thus, it is natural to ask the following question.

\begin{problem}
Are configuration spaces of graphs with sinks also aspherical?
\end{problem}

The author must confess that he has not checked whether the classical arguments carry over into this setting, but see \cite{Ramos:CSGCTPC} for the case in which every vertex is a sink. 

Finally, we note that, in some applications, one is not merely interested in abstract motion planning, but moreso in \emph{efficient} motion planning. From this point of view, it is relevant to recall that a configuration space of a graph is homotopy equivalent to (two different!) CAT(0) cube complexes \cite{Abrams:CSBGG,Swiatkowski:EHDCSG}, which arrive equipped with an intrinsic metric and notion of geodesic.

\begin{problem}
Calculate the geodesic complexity \cite{Recio-Mitter:GCMP} of the cubical models of Abrams and {\'{S}}wi\k{a}tkowski. Does an analogue of Farber's conjecture hold here?
\end{problem}

\bibliographystyle{amsalpha}
\bibliography{references}

\end{document}